\documentclass[a4paper,11pt]{article}

\usepackage{amssymb,amsmath, amsthm}
\usepackage{enumerate} 
\usepackage{enumitem} 
\usepackage{graphicx}
\usepackage{xcolor}
\usepackage{fullpage}

\newcommand{\ck}[1]{C_{k \geq #1}}
\newcommand{\ns}{\not\subseteq}
\newcommand{\F}{\mathcal{F}}
\newcommand{\C}{\mathcal{C}}

\newtheorem{theorem}{Theorem}[section]
\newtheorem{lemma}[theorem]{Lemma}

\def\longbox#1{\parbox{0.85\textwidth}{#1}}

\begin{document}

\title{Decomposition techniques applied to the \\ Clique-Stable set Separation problem}

\author{%
    Nicolas Bousquet\thanks{Laboratoire G-SCOP, CNRS, Univ. Grenoble Alpes, Grenoble, France.\newline
    The authors are partially supported by ANR project STINT (reference ANR-13-BS02-0007).}
    \and
    Aur\'elie Lagoutte\footnotemark[1] \and
    Fr\'ed\'eric Maffray\footnotemark[1] \and
    Lucas Pastor\footnotemark[1]}

\date{\today}

\maketitle

\begin{abstract}
In a graph, a Clique-Stable Set separator (CS-separator) is a family
$\mathcal{C}$ of cuts (bipartitions of the vertex set) such that for
every clique $K$ and every stable set $S$ with $K \cap S = \emptyset$,
there exists a cut $( W,W')$ in $\mathcal{C}$ such that $K \subseteq
W$ and $S \subseteq W'$.  Starting from a question concerning extended
formulations of the Stable Set polytope and a related complexity
communication problem, Yannakakis \cite{Y} asked in 1991 the following
questions: does every graph admit a polynomial-size CS-separator?  If
not, does every perfect graph do?  Several positive and negative
results related to this question were given recently.  Here we show
how graph decomposition can be used to prove that a class of graphs
admits a polynomial CS-separator.  We apply this method to apple-free
graphs and cap-free graphs.

    \noindent
    \textit{Keywords}: clique-stable set separation, extended formulation, stable set polytope, graph decomposition, apple-free graphs.
\end{abstract}

\section{Introduction}
Let $G$ be any simple finite undirected graph without loop.  In 1991
Yannakakis \cite{Y} asked for the existence of a compact extended
formulation of the Stable Set polytope of $G$, i.e., the existence of
a simpler polytope in higher dimension whose projection is the Stable
Set polytope of $G$ (the Stable Set polytope $STAB(G)$ is the convex
hull in $\mathbb{R}^{|V(G)|}$ of the characteristic vectors of all the
stable sets of $G$).  Such a simpler polytope would correspond to a
linear program with extra variables to solve the Maximum Weighted
Stable Set problem in $G$.  However, there is no good explicit
description of the Stable Set polytope for general graphs.
Consequently, Yannakakis considered this problem in restricted classes
of graphs on which a description of the Stable Set polytope with
equalities and inequalities is known.  This is the case for perfect
graphs: the stable set polytope is exactly the polytope described by
the clique inequalities ($\sum_{v\in K} x_v\leq 1$, for every clique
$K\subseteq V(G)$) together with the non-negativity constraints.  This
case highlights a combinatorial object which provides a combinatorial
lower bound (the so-called \emph{rectangle covering bound}) for the
size of any extended formulation for $STAB(G)$.  

Yannakakis then raised the question below for both perfect graphs and
general graphs.  We state this question in its graph-theoretical
version but it can also be stated as a communication complexity
problem.  In a graph $G$, a \emph{cut} is a partition $(W,W')$ of
$V(G)$ into two subsets.  Given two disjoint subsets of vertices $K$
and $S$, a cut $(W, W')$ \emph{separates} $K$ and $S$ if $K \subseteq
W$ and $S \subseteq W'$.  A \emph{Clique-Stable Set separator
(CS-separator)} is a family $\mathcal{C}$ of cuts such that, for every
clique $K$ and every stable set $S$ with $K \cap S = \emptyset$, there
is a member of $\mathcal{C}$ that separates $K$ and $S$.  The
\emph{size} of a CS-separator is the number of cuts contained in the
family.  Yannakakis's question is the following:
\begin{equation*}
\longbox{{\it Does there exist a Clique-Stable set separator
consisting of polynomially many cuts?}}
\end{equation*}
Yannakakis proved in~\cite{Y} that every graph on $n$
vertices admits a CS-Separator of size $n^{\mathcal{O}(\log n)}$.  He
also proved that the existence of a polynomial-size CS-separator is a
necessary condition for the existence of a polynomial (more frequently
called \emph{compact}) extended formulation of the Stable Set
polytope.

The existence or not, for every graph $G$, of a polynomial-size
CS-separator and/or a compact extended formulation for $STAB(G)$ has
been a big open question since 1991; see e.g.~Lov\'asz's survey
\cite{L}.  It has seen a renewed interest since 2012 with the
following two results: on the one hand, Fiorini~\emph{et al.}
\cite{Fio} proved that some graphs do not admit a compact extended
formulation for $STAB(G)$; on the other hand, Huang and Sudakov
\cite{HuangS12} proved the first superlinear lower bound for the size
of a CS-Separator in general.  This lower bound was then improved in
\cite{Am} and then in \cite{ShigetaA}.  Eventually, G\"o\"os \cite{G}
proved in 2015 that some graphs do not admit a polynomial-size
CS-Separator.  Hence, Yannakakis's two questions have been finally
given a negative answer in the general case.  However, determining
which graph classes admit a polynomial CS-Separator
remains a widely open problem. %
A class $\mathcal{C}$ of graphs is said to \emph{have the polynomial
CS-Separation property} if there exists a polynomial $P$ such that
every graph $G\in \mathcal{C}$ admits a CS-Separator of size
$P(|V(G)|)$.  We further say that $\mathcal{C}$ has the
$\mathcal{O}(n^c)$-CS-Separation property if $P(n)$ is a
$\mathcal{O}(n^c)$ in the previous definition.

Given a family of graphs $\mathcal{F}$, a graph $G$ is
\emph{$\mathcal{F}$-free} if no induced subgraph of $G$ is isomorphic
to a member of $\mathcal{F}$.  If $\mathcal{F}$ is composed of only
one element $F$, we say that $G$ is $F$-free.  As usual we let $P_k$
and $C_k$ denote respectively the chordless path and chordless cycle
on $k$ vertices.  A \emph{hole} is any chordless cycle on at least
four vertices.  An \emph{antihole} is the complementary graph of a
hole.

It is easy to see that when the number of maximal cliques, or the
number of maximal stable sets, is a polynomial in the size of the
graph, then there exists a polynomial-size CS-Separator.  In
particular, this happens when $\omega(G)$ or $\alpha(G)$ is bounded by
a constant, and by \cite{A}, this also happens for $C_4$-free graphs.

Concerning subclasses of perfect graphs, Yannakakis \cite{Y} proved
that comparability graphs and chordal graphs have the polynomial
CS-separator property, and they even have compact extended
formulations.  Since a CS-Separator of the complement graph can be
trivially obtained from a CS-Separator of the graph, the polynomial
CS-Separation property also holds for the complements of comparability
graphs and complements of chordal graphs.  Moreover, Lagoutte et
al.~\cite{LT} proved that perfect graphs with no balanced
skew-partition have the quadratic CS-Separation property.  Another
interesting subclass of perfect graphs is the class of \emph{weakly
chordal graphs}, which are the graphs that contain no hole of length
at least~$5$ and no antihole of length at least~$5$ \cite{hay}.
Weakly chordal graphs have the polynomial CS-separator property: this 
follows from the main result of~\cite{BBT} and Theorem 12 of~\cite{BLT}. 

Leaving the domain of perfect graphs, Yannakakis proved that
$t$-perfect graphs have the polynomial CS-Separation property, and
even that they have compact extended formulations for the Stable Set
polytope. %
Bousquet et al.~\cite{BLT} showed that, asymptotically almost surely,
there exists a polynomial size CS-separator for $G$ if $G$ is picked
at random.  Furthermore, they showed that $H$-free graphs, where $H$
is any split graph (i.e., a graph whose vertex set can be partitioned
into a clique and a stable set), $P_5$-free graphs, and
($P_k,\overline{P_k}$)-free graphs all satisfy the polynomial
CS-Separation property.  Furthermore, for any $k$ the class of
graphs that contain no hole of length at least~$k$ and no antihole of
length at least~$k$ have the ``strong Erd\H{o}s-Hajnal
property''~\cite{BBT} (note that weakly chordal graphs correspond to
the case $k=5$), and every hereditary class of graphs that has the
strong Erd\H{o}s-Hajnal property has the polynomial CS-Separation
property \cite{BLT}.
With the results from~\cite{GKPP} in which the authors devise a polynomial-time
algorithm for the Maximum Weight Independent Set problem in $P_6$-free graphs
and by the same techniques as for the $P_5$-free graphs and Corollary 15
from~\cite{BLT}, the class of $P_6$-free graphs also has the polynomial
CS-Separation property.


\paragraph{Our contribution.}
The goal of this paper is to show that graph decomposition can be used
to prove that a class of graphs has the polynomial CS-Separation
property.  In the past decades many decomposition theorems have been
proposed for graph classes, and especially for subclasses of perfect
graphs.  The existence of such a decomposition was used for example
in~\cite{LT} to prove that a certain subclass of perfect graphs has
the polynomial CS-Separation property.  Conforti~\emph{et
al.}~\cite{ConfortiSTAB} have recently followed a similar approach as
we do in this paper: they study how graph decomposition can help to
prove that the Stable Set polytope admits a compact extended
formulation, when this property holds in ``basic" graphs.

Here we investigate further the relation between graph decomposition
and CS-separators.  To this end, we define in Section~\ref{sec:tree}
the general setting of a decomposition tree and highlight in
Section~\ref{sec:validdec} some decomposition rules that behave well
with respect to CS-Separators.  We then apply this technique in
Section~\ref{sec:applications} to two classes of graphs where existing
decomposition theorems provide both ``friendly" decompositions on one
hand, and ``basic" graphs already having the polynomial CS-Separation
property on the other hand.  These two classes are apple-free graphs
and cap-free graphs (see Section~\ref{sec:applications} for
definitions).  Moreover, we briefly mention diamond-wheel-free graphs
and $k$-windmill graphs.  In Section~\ref{sec:limit}, we show that the
existence of a recursive star-cutset decomposition for a graph $G$
cannot be enough to ensure a polynomial size CS-separator.

\paragraph{Notations and definitions.} For any vertex $v \in
V(G)$, we denote by $N(v) = \{u \in V(G) \mid uv \in E(G)\}$ the set
of vertices adjacent to $v$, called the \emph{neighborhood} of $v$.
The \emph{closed neighborhood} of $v$, denoted by $N[v]$ is the set $\{v\}\cup N(v)$.
The set $V(G)\setminus N[v]$ of vertices not adjacent to $v$ is called the
\emph{anti-neighborhood} of $v$.  
For any $S \subseteq V(G)$ we denote
by $G[S]$ the \emph{induced subgraph} of $G$ with vertex set $S$.
Given a class of graphs $\mathcal{C'}$, a graph $G$ is \emph{nearly
$\mathcal{C'}$} if for every $v\in V(G)$, we have $G\setminus N[v]\in
\mathcal{C'}$.

For two sets $A,B \subseteq V(G)$, we say that $A$ is \emph{complete}
to $B$ if every vertex of $A$ is adjacent to every vertex of $B$, and
we say that $A$ is \emph{anticomplete} to $B$ if no vertex of $A$ is
adjacent to any vertex of $B$.  A \emph{module} is a set $M \subseteq
V(G)$ such that every vertex in $V(G) \setminus M$ is either complete
to $M$ or anticomplete to $M$.  A module $M$ is \emph{trivial} if
either $|M| = 1$ or $M = V(G)$.  A graph is \emph{prime} if all its
modules are trivial.

A \emph{vertex cut} $C$ is a subset of vertices such that $G[V
\setminus C]$ is disconnected.  We say that $C$ is a \emph{minimal
vertex cut} if $C$ does not contain any other vertex cut.  A
\emph{clique-cutset} is a vertex cut that induces a clique.  A graph
that does not contain any clique-cutset is called an \emph{atom}.

A graph is \emph{chordal} if it is does not contain as
an induced subgraph any $C_k$ for $k \geq 4$.   A class of graphs
is \emph{hereditary} if it is closed under taking induced subgraphs.



\section{General setting of the method}
\label{sec:tree}

Decomposition trees have been widely used to solve algorithmic
problems on graphs.  The idea consists in
breaking the graph into smaller parts in a divide-and-conquer
approach.  The problem is solved on the smaller parts and then these
partial solutions are combined in some way into a solution for the
whole graph.  To solve the problem on the smaller parts, one
recursively breaks them down into smaller and smaller graphs until one
reaches graphs that are simple enough -- they form the leaves of the
decomposition tree.  Then, from solutions on the leaves, one follows a
bottom-up approach to build a solution for the parents of the leaves,
and ultimately for the whole graph.  The decomposition rules depend
both on the class of graphs under study and on the problem under
consideration.  We try to present here the most general setting.

A \emph{decomposition} of a graph $G$ is a pair $(G_1, G_2)$ where
$G_1$ and $G_2$ are proper induced subgraphs of $G$ ($G_1$ and $G_2$ 
are often   called \emph{blocks of the decomposition} in the literature).
%
A decomposition $(G_1, G_2)$ is \emph{valid (with respect to the
CS-Separation)} if, given a CS-Separator of size $f_1$ of $G_1$ and a
CS-Separator of size $f_2$ of $G_2$, there exists a CS-Separator of
size $f_1+f_2$ of $G$.  We discuss in detail in
Section~\ref{sec:validdec} which graph decompositions are valid.
 
Let $\mathcal{C'}$ be a class of graphs.  Given a graph $G$, a rooted
binary tree $T$ is a \emph{valid decomposition tree for $G$ with
leaves in $\mathcal{C'}$} if the following conditions hold (an example is described in Figure~\ref{fig: decompo tree}):
 \begin{itemize}
\item 
There exists a map $\varphi:V(T)\to \mathcal{P}(V(G))$ that associate
with every $t\in V(T)$ a subset of vertices of $G$.  We say that
$G[\varphi(t)]$ is the \emph{subgraph at node $t$}\footnote{As is
usually the case to avoid confusion, vertices of the tree
decomposition will be called nodes.}.
\item 
If $r$ is the root of $T$, then $\varphi(r)=V(G)$, i.e., $G$ itself is
the subgraph at node $r$.  
\item 
If $f$ is a leaf of $T$, then $G[\varphi(f)]\in \mathcal{C'}$.
\item 
If $t$ is a node of $V(T)$ with children $s, s'$, then
$(G[\varphi(s)], G[\varphi(s')])$ is a valid decomposition of
$G[\varphi(t)]$.
 \end{itemize}
We easily obtain the following:
 
\begin{lemma} \label{lem: CS-Sep from poly valid decompo tree}
Let $\mathcal{C'}$ be a class of graphs 
having the $\mathcal{O}(n^c)$-CS-Separation property, 
 for some
constant $c>0$.  Let $G$ be a graph on $n$ vertices admitting a valid
decomposition tree $T$ with leaves in $\mathcal{C'}$, and $L(T)$ be
the number of leaves of $T$.  Then there exists a CS-Separator for $G$
of size $\mathcal{O}(L(T)\cdot n^c)$.
 \end{lemma}
 
 \begin{proof}
Let $\ell_1, \ldots , \ell_{L(T)}$ be the leaves of $T$.  Then for
each $i$, $G[\varphi(\ell_i)]$ admits a CS-Separator of size $f_i$
with $f_i=\mathcal{O}(|\varphi(\ell_i)|^c)= \mathcal{O}(n^c)$ (since
$|\varphi(\ell_i)|\leq n$).  By definition of a valid decomposition,
we obtain that $G$ admits a CS-Separator of size $\sum_{i=1}^{L(T)}
f_i= \mathcal{O}(L(T)\cdot n^c)$.
 \end{proof}
 
>From Lemma~\ref{lem: CS-Sep from poly valid decompo tree}, we can
design a proof strategy.  Let $\mathcal{C}$ be a class of graphs in
which we want to prove that the polynomial CS-Separation property holds.  We need
to reach the following intermediate goals:
\begin{itemize}
\item 
Find a suitable class $\mathcal{C'}$ having the polynomial CS-Separation
property, 
\item 
Prove that every $G\in \mathcal{C}$ admits a valid decomposition tree
$T(G)$ with leaves in $\mathcal{C}'$, and 
\item 
Prove that $T(G)$ has size polynomial in $|V(G)|$.
\end{itemize}

In practice, the last item is often the hardest to obtain.  However,
one way to get over this problem consists in labeling the nodes of the
tree $T(G)$.  Let $\mathcal{S}$ be a polynomial set of subsets of
$V(G)$ (e.g. $\mathcal{S}$ contains only subsets of at most $k$
vertices for some fixed $k$).  A \emph{$\mathcal{S}$-labeling} of
$T(G)$ is a map $\ell:I(T)\to \mathcal{S}$ where $I(T)$ is the set of
internal nodes of $T$, with the condition that $\ell(t)\subseteq
\varphi(t)$ for every $t\in T$.  In other words, the label of a node
$t$ must contain only vertices from the subgraph at node $t$ (see Figure \ref{fig: decompo tree} for an illustration of the definition).  The
$\mathcal{S}$-labeling is \emph{injective} if $\ell(t)\neq \ell(t')$
whenever $t\neq t'$.  The existence of an injective
$\mathcal{S}$-labeling ensures that $|I(T)|\leq |\mathcal{S}|$ and
consequently $|V(T)|\leq 2|I(T)|+1\leq 2|\mathcal{S}|+1$ since the
number of leaves of a binary tree is at most its number of internal
vertices plus one.  The following lemma provides sufficient conditions
for a $\mathcal{S}$-labeling to be injective.

\begin{lemma} \label{lem: injective labeling}
Let $G$ be a graph admitting a valid decomposition tree $T(G)$ and
$\ell$ be a $\mathcal{S}$-labeling of $T(G)$.  If for every node $t$
with children $s$, $s'$ we have:
\begin{enumerate}
\item
$\ell(t) \ns \varphi(s)$ and $\ell(t) \ns \varphi(s')$, and 
\item
No member of $\mathcal{S}$ is included in $\varphi(s)\cap
\varphi(s')$, unless $s$ or $s'$ is a leaf,
\end{enumerate}
then $\ell$ is injective.
\end{lemma}
Informally, the first condition ensures that the label we choose for node $t$ is "broken" at
each decomposition (i.e. it does not appear in any children), and the second condition ensures 
that no potential  label is
duplicated, i.e. no potential label appears in both children.

\begin{proof}
Assume by contradiction that $\ell$ is not injective and let $t\neq
t'$ be two internal nodes such that $\ell(t)=\ell(t')$.  Suppose first
that $t$ is an ancestor of $t'$.  Let $s$ and $s'$ be the children of
$t$, where $s'$ is an ancestor of $t'$ (possibly $s'=t'$).  By
definition, $\ell(t')\subseteq \varphi(t')\subseteq \varphi(s')$.  But
$\ell(t)$ is not a subset of $\varphi(s')$, a contradiction with
$\ell(t)=\ell(t')$.  Now, suppose that none of $t,t'$ is an ancestor
of the other.  Let $t_0$ be the closest common ancestor of $t$ and
$t'$, and let $s$ (resp.~$s'$) be the child of $t_0$ which is an
ancestor of $t$ (resp.~of $t'$).  Clearly none of $s, s'$ is a leaf.
Since $\ell(t)\subseteq \varphi(t) \subseteq \varphi(s)$ and similarly
$\ell(t')\subseteq \varphi(t') \subseteq \varphi(s')$, we obtain that
$\varphi(s)\cap \varphi(s')$ contains $\ell(t)=\ell(t')\in
\mathcal{S}$: a contradiction to the second condition of the lemma.
\end{proof}


\begin{figure}\label{fig: decompo tree}
\begin{center}
\includegraphics[scale=1.1]{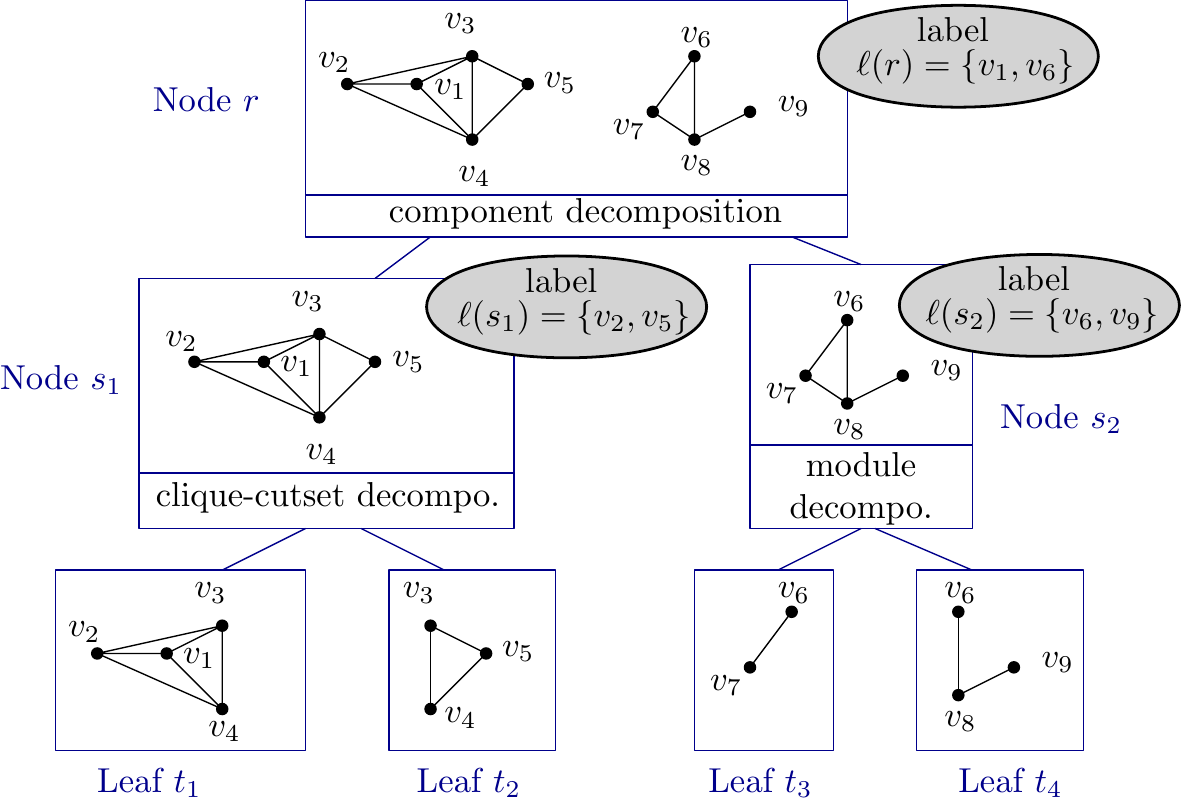}
\end{center}

\caption{Example of a valid decomposition tree $T$ for a graph $G$ on 9 vertices (displayed in the upper box), with leaves in $\mathcal{C}$, where $\mathcal{C}$ denote the union of triangle-free graphs and complete graphs. 
In this example, $T$ has 3 internal nodes and 4 leaves, with 
$\varphi(r)=\{v_1, \ldots, v_9\}$, 
$\varphi(s_1)=\{v_1, \ldots, v_5\}$, 
$\varphi(s_2)=\{v_6, \ldots, v_9\}$, 
$\varphi(t_1)=\{v_1, \ldots, v_4\}$, 
$\varphi(t_2)=\{v_3, \ldots, v_5\}$, 
$\varphi(t_3)=\{v_6,  v_7\}$, and
$\varphi(t_4)=\{v_6, \ldots, v_9\}$.
Moreover, this tree is equipped with a $\mathcal{S}$-labeling $\ell$, where $\mathcal{S}$ denotes the set of all pairs of non-adjacent vertices. Observe that $\ell$ satisfies the condition of Lemma \ref{lem: injective labeling}.
See section \ref{sec:validdec} for details about the valid decompositions that are used.
}
\end{figure}

\section{Valid decompositions}
\label{sec:validdec}

The core of the paper consists in showing that many usual graph
decompositions are actually valid decompositions with respect to the CS-Separation.  

\begin{itemize}
\item 
Suppose that $G$ is not connected.  So $V(G)$ admits a partition
$(A,B)$ where $A$ is anticomplete to $B$ and both $A$ and $B$ are non-empty.  Then $(G[A], G[B])$ is a
\emph{component decomposition} of $G$.  (Note that $G[A]$ and $G[B]$
may not be connected, but this does not matter.)

\item 
Suppose that $\overline{G}$ is not connected.  So $V(G)$ admits a
partition $(A,B)$ where $A$ is complete to $B$, and both $A$ and $B$ are non-empty.  Then $(G[A], G[B])$
is an \emph{anticomponent decomposition} of $G$.


\item 
More generally, suppose that $V(G)$ admits a partition $(A, C, B)$
such that $A$ and $B$ are not empty and $A$ is anticomplete to $B$.
Then $(G[A\cup C], G[B\cup C])$ is a \emph{cutset decomposition} of
$G$.  If $C$ is a clique, we say that it is a \emph{clique-cutset
decomposition}.

\item 
Suppose that $G$ admits a non-trivial module $M$.  Pick any $m\in M$.
Then $(G[M], G[\{m\}\cup (V(G)\setminus M)])$ is a \emph{module
decomposition} of $G$.

\item 
Let $\mathcal{C'}$ be a class of graphs.  Suppose that there is a
vertex $v\in V(G)$ such that $G[N(v)]\in \mathcal{C'}$.  Then
$(G\setminus v, G[N(v)])$ is a \emph{$\mathcal{C'}$-neighborhood
decomposition} (at vertex $v$).

\item 
Similarly, suppose that there exists $v\in V(G)$ such that
$G[V(G)\setminus N[v]]\in \mathcal{C'}$.  Then $(G\setminus v,
G\setminus N[v])$ is a \emph{$\mathcal{C'}$-antineighborhood
decomposition} (at vertex $v$).

\item 
An \emph{amalgam split} of $G$ (see \cite{BF}) is a partition of
$V(G)$ into five sets $B_1, A_1, C, A_2, B_2$ with $A_1, A_2$ both non-empty, $|A_1\cup B_1|\geq 2$, $|A_2\cup B_2|\geq 2$ satisfying the following conditions:
$C$ is a
clique, $C$ is complete to $A_1\cup A_2$, $A_1$ is complete to $A_2$,
$B_1$ is anticomplete to $A_2\cup B_2$, and $B_2$ is anticomplete to
$A_1\cup B_1$.   Pick any $a_1\in A_1$ and $a_2\in A_2$.
Then $(G[B_1\cup A_1\cup C\cup\{a_2\}], G[\{a_1\}\cup C \cup A_2\cup
B_2])$ is an \emph{amalgam decomposition} of $G$.
\end{itemize}

Observe that some of these decompositions were introduced in the
context of perfect graphs, as partial attempts to prove Berge's Strong
Perfect Graph Conjecture by proving that minimal counter-examples to
the conjecture do not admit such a decomposition.  Since it is not
known whether the polynomial CS-Separation property holds in perfect graphs, it
makes sense to follow these steps and to carefully study the
behavior of these decompositions with respect to CS-Separators.

We now prove that each of these decompositions is valid.  It is worth
noting that if $G$ admits a so-called \emph{2-join}, then a valid decomposition can
be found too.  We will not use it in our applications in Section~\ref{sec:applications},
but we refer the interseting reader to \cite{LT} for more details.

\begin{lemma}\label{lem:valid decompo}
Let $G$ be a graph and let $(G_1, G_2)$ be a decomposition of $G$ which is either a component decomposition, or an anticomponent decomposition, or a cutset decomposition, or a module decomposition, or a $\mathcal{C'}$-neighborhood decomposition, or a $\mathcal{C'}$-antineighborhood decomposition, or an
amalgam decomposition. Then this decomposition is valid.
\end{lemma}
\begin{proof}
We follow the notations used in the definitions above.  Let
$\mathcal{F}_1$ be a CS-separator of $G_1$ and $\mathcal{F}_2$ be a
CS-separator of $G_2$, of respective sizes $f_1$ and $f_2$.  In each
case, we will define a family $\mathcal{F}$ of cuts of $G$ of size
$f_1+f_2$ and prove that $\mathcal{F}$ is a CS-separator for $G$.
    
\medskip

\noindent
{\it Cutset decomposition}: Here $G$ has a partition $(A,C,B)$
where $A$ and $B$ are not empty and $A$ is anticomplete to $B$, and
$G_1=G[A\cup C]$, and $G_2=G[B\cup C]$.  For each cut $(W, W')\in
\mathcal{F}_1$, we put $(W, W'\cup B)$ in $\mathcal{F}$; and similarly
for each cut $(W, W')\in \mathcal{F}_2$, we put $(W, W'\cup A)$ in
$\mathcal{F}$.  Clearly $\mathcal{F}$ is a set of cuts of $G$.  Let us
check that $\mathcal{F}$ is a CS-Separator of $G$.  Pick any clique
$K$ and any stable set $S$ of $G$, such that $K\cap S=\emptyset$.
Note that $K$ cannot intersect both $A$ and $B$ since $A$ is
anticomplete to $B$.  Up to symmetry, let us assume that $K\cap
B=\emptyset$, so $K \subseteq A\cup C$.  Then there is a cut $(W, W')
\in \mathcal{F}_1$ that separates $K$ and $S \cap (A\cup C)$ (possibly
empty).  Then the cut $(W, W' \cup B)$ separates $K$ and $S$ in $G$
and it is a member of $\mathcal{F}$.

\medskip

\noindent
{\it Component decomposition}: This is similar to a cutset
decomposition with $C$ being empty.

\medskip

\noindent
{\it Anticomponent decomposition}: Since a CS-separator of a graph
$G$ is also a CS-separator of its complement $\overline{G}$ (up to
swapping both sets in the cuts), $\mathcal{F}_1$ and
$\mathcal{F}_2$ also are CS-separators of the complement of 
$G_1$ and $G_2$, respectively.  The first point ensures that the
complement of $G$ has a CS-separator of size $f_1+f_2$, and then so
does $G$.

\medskip

\noindent
{\it Module decomposition}: Here $G$ has a partition $(M, A, B)$
where $A$ is complete to $M$, and $B$ is anticomplete to $M$, and
$G_1=G[M]$, and $G_2=G[\{m\}\cup A\cup B]$ for some $m\in M$.  For each
cut $(W, W')\in \mathcal{F}_1$, we put the cut $(W\cup A, W' \cup B)$
in $\mathcal{F}$.  For each cut $(W, W') \in \mathcal{F}_2$, if $m\in
W$ then we put $(W \cup M, W')$ in $\mathcal{F}$, otherwise we put
$(W, W'\cup M)$ in $\mathcal{F}$.  Let us check that $\mathcal{F}$ is
a CS-Separator of $G$.  Pick any clique $K$ and any stable set $S$ of
$G$, such that $K\cap S=\emptyset$.  First, suppose that $K\subseteq
M\cup A$ and $S\subseteq M\cup B$.  Then there is a cut $(W, W')\in
\mathcal{F}_1$ separating $K\cap M$ and $S\cap M$, and the
corresponding cut $(W\cup A, W'\cup B)\in \mathcal{F}$ separates $K$
and $S$.  So $K\nsubseteq M\cup A$ or $S\nsubseteq M\cup B$.  
This implies that $M$ cannot intersect both $K$ and $S$.
 Let $K_2=K$ if $K\cap M=
\emptyset$ and $(K \setminus M) \cup \{m\}$ otherwise.  Similarly let
$S_2=S$ if $S\cap M= \emptyset$ and $(S \setminus M) \cup \{m\}$
otherwise.  Note that $K_2$ and $S_2$ are respectively a clique and a
stable set in $G_2$.  So there exists a cut $(W,W') \in
\mathcal{F}_2$ that separates $K_2$ and $S_2$.  If $m \in W$, then $(W
\cup M,W')$ is a member of $\mathcal{F}$ that separates $K \cup M$ and
$S$, and consequently $K$ and $S$.  If $m \in W'$, then $(W,W' \cup
M)$ is a member of $\mathcal{F}$ that separates $K$ and $S \cup M$,
and consequently $K$ and $S$.

\medskip

\noindent
{\it $\mathcal{C'}$-neighborhood decomposition}: Here $G$ has a partition $(\{v\},
A, B)$ where $A=N(v)$, and $G[A]\in \mathcal{C}'$, and $G_1 = G[V(G)
\setminus \{v\}]$, and $G_2 = G[N(v)]$.  For each cut $(W, W') \in
\mathcal{F}_1$, we put $(W, W' \cup \{v\})$ in $\mathcal{F}$.  For
each cut $(W, W') \in \mathcal{F}_2$, we put $(W \cup \{v\}, W' \cup
B)$ in $\mathcal{F}$.  Let us check that $\mathcal{F}$ is a
CS-Separator of $G$.  Pick any clique $K$ and any stable set $S$ of
$G$, such that $K\cap S=\emptyset$.  First, suppose that $v \notin K$,
hence $K \subseteq A \cup B$.  Then there is a cut $(W, W') \in
\mathcal{F}_1$ that separates $K$ and $S \cap (A\cup B)$, and
consequently the cut $(W, W' \cup \{v\})$ is a member of $\mathcal{F}$
that separates $K$ and $S$ in $G$.  Now, assume that $v \in K$, hence
$K \subseteq N[v]$ and $S \subseteq A \cup B$.  Then there is a cut
$(W, W') \in \mathcal{F}_2$ that separates $K \cap A$ and $S \cap A$,
and the cut $F = (W \cup \{v\}, W' \cup B)$ is a member of
$\mathcal{F}$ that separates $K$ and $S$ in $G$.
	
\medskip

\noindent
{\it $\mathcal{C'}$-anti-neighborhood decomposition}: the complement of a graph with a $\mathcal{C'}$-antineighborhood
decomposition on $v$ is a graph with a $\overline{\mathcal{C'}}$-neighborhood decomposition on
$v$, so we apply the same construction as above in the complement graph.  
 
\medskip

\noindent
{\it Amalgam decomposition}: Here $G$ has a partition $(B_1, A_1, C,
A_2, B_2)$ where $C$ is a clique, $C$ is complete to $A_1 \cup A_2$,
$A_1$ is complete to $A_2$, $B_1$ is anticomplete to $A_2 \cup B_2$,
and $B_2$ is anticomplete to $A_1 \cup B_1$, and $G_1 = G[A_1 \cup B_1
\cup C \cup \{a_2\}]$, and $G_2 = G[A_2 \cup B_2 \cup C \cup \{a_1\}]$
for some $a_1 \in A_1$ and $a_2 \in A_2$.  For each cut $(W, W') \in
\F_1$, if $a_2 \in W$ we put the cut $(W \cup A_2, W' \cup B_2)$ in
$\F$, otherwise we put $(W, W' \cup A_2 \cup B_2)$ in $\F$.  Similarly, for each
cut $(W, W') \in \F_2$, if $a_1 \in W$ we put the cut $(W \cup A_1, W'
\cup B_1)$ in $\F$, otherwise we put $(W, W' \cup A_1 \cup B_1)$ in
$\F$.  Let us check that $\mathcal{F}$ is a CS-Separator of $G$.  Pick
any clique $K$ and any stable set $S$ of $G$, such that $K\cap
S=\emptyset$.  First, suppose that $K \cap (B_1 \cup B_2) \neq
\emptyset$.  Up to symmetry, we may assume that $K \cap B_1 \neq
\emptyset$ and let $S_1 = S\setminus B_2$ if $S \cap (A_1 \cup C) \neq \emptyset$
and $S_1 = (S \setminus (A_2\cup B_2)) \cup \{a_2\}$ otherwise.  Note that $K$
and $S_1$ are respectively cliques and stable sets in $G_1$.  So there
exists a cut $(W, W') \in \F_1$ that separates $K$ and $S_1$ in $G_1$.
If $S \cap (A_1 \cup C) = \emptyset$, then $a_2 \in S_1$, hence $a_2
\notin W$, so the cut $(W, W' \cup A_2 \cup B_2)$ is a member of $\F$
that separates $K$ and $S$ in $G$.  On the other hand, if $S \cap (A_1
\cup C) \neq \emptyset$, then $S \cap A_2 = \emptyset$, so whether or
not $a_2$ is in $W$, both possible cuts $(W \cup A_2, W' \cup B_2)$
and $(W, W' \cup A_2 \cup B_2)$ separate $K$
and $S$ in $G$, and one of them is a member of $\F$.  We may assume now that $K \cap (B_1 \cup B_2) =
\emptyset$.  First, suppose that $S \cap (A_1 \cup A_2) \neq
\emptyset$, so, up to symmetry, we may assume that $S \cap A_1 \neq
\emptyset$ and let $K_1 = (K \setminus A_2) \cup \{a_2\}$.  Note that
$K_1$ is a clique in $G_1$.  So there exists a cut $(W, W') \in \F_1$
that separates $K_1$ and $S \cap V(G_1)$.  Hence, the cut $(W \cup
A_2, W' \cup B_2)$ is a member of $\F$ that separates $K$ and $S$ in
$G$.  Finally, suppose that $S \cap (A_1 \cup A_2) = \emptyset$ and
let $K_1 = (K \setminus A_2) \cup \{a_2\}$.  Note that $K_1$ is a
clique in $G_1$.  So there exists a cut $(W, W') \in \F_1$ that
separates $K_1$ and $S \cap V(G_1)$.  Hence, the cut $(W \cup A_2, W'
\cup B_2)$ is a member of $\F$ that separates $K$ and $S$ in $G$.
\end{proof}

Observe now that degeneracy can be seen as a kind of decomposition:

\begin{theorem}
Let $s>0$ be a constant and let $G$ be a graph on $n$ vertices admitting an ordering $v_1, \ldots, v_n$ of its vertices such that 
$|N(v_i)\cap\{v_{i+1}, \ldots, v_n\}|\leq n^{s/\sqrt{\log n}}$ for every $i \in \{1, \ldots, n - 1\}$. Then $G$ admits a CS-Separator of size $n^{\mathcal{O}(1)}$.
\end{theorem}

\begin{proof}
Let $\mathcal{C'}$ be the class of all graphs of size at most $n^{s/\sqrt{\log n}}$. By \cite{Y}, every graph with $m$ vertices admits a CS-Separator of size $m^{\mathcal{O}(\log m)}$, i.e. of size at most $m^{c\log m}$ for some constant $c$. Hence every graph  $H\in\mathcal{C'}$ admits a CS-Separator of size 

$$\left(n^{s/\sqrt{\log n}}\right)^{c\log\left(n^{s/\sqrt{\log n}}\right)}=n^{cs^2} \qquad
\text{where } n \text{ still denotes }|V(G)| \ .$$
Now we recursively decompose $G[v_i, \ldots, v_n]$ with a $\mathcal{C'}$-neighborhood decomposition at vertex $v_i$, namely $(G[\{v_{i+1}, \ldots , v_n\}], N(v_i) \cap\{v_{i+1}, \ldots, v_n\})$, which is possible by degeneracy hypothesis. We obtain a valid decomposition tree $T$ for $G$ with leaves in $\mathcal{C'}$, such that the right son of every internal node of $T$ is a leaf, and the height of $T$ is $n-1$. This proves that $|V(T)|$ is linear in $n$. By Lemma \ref{lem: CS-Sep from poly valid decompo tree}, $G$ admits a CS-Separator of size $\mathcal{O}(n^{cs^2+1})$ which is a $n^{\mathcal{O}(1)}$.
\end{proof}

We finish this section by studying two types of situations that are
not really, or not at all, decompositions but nevertheless are nice to
encounter when one tries to obtain polynomial-size CS-separators.

\begin{theorem}\label{thm:prime-atoms-nearly-Cp}
Let $\mathcal{C}$ be a hereditary class of graphs and $c\geq 1$.
Assume that $\mathcal{C}'$ is a class of graphs 
having the $\mathcal{O}(n^c)$-CS-Separation property.
 If every
prime atom of $\mathcal{C}$ is nearly-$\mathcal{C}'$, then 
$\mathcal{C}$ has the $\mathcal{O}(n^{c+3})$-CS-Separation property.
\end{theorem}
\begin{proof}
We follow the method introduced in Section~\ref{sec:tree} by
recursively decomposing every member $G$ of $\mathcal{C}$ along
components, anticomponents, modules, clique-cutset, or
$\mathcal{C'}$-antineighborhood until reaching graphs of
$\mathcal{C}'$ or cliques.

Let $\mathcal{K}$ be the class of complete graphs.  For every
$G\in\mathcal{C}$, we recursively define a valid decomposition tree
$T(G)$ with leaves in $\mathcal{C'}\cup \mathcal{K}$ as follows.  If
$G\in\mathcal{C}'\cup \mathcal{K}$, then the root $r$ is the only node
of $T(G)$.  Otherwise, we use a valid decomposition $(G_1, G_2)$ of
$G$ as described below and $T(G)$ is obtained from $T(G_1)$ and
$T(G_2)$ by connecting the root $r$ to the respective roots of
$T(G_1)$ and $T(G_2)$.  Moreover the map $\varphi:V(T)\to
\mathcal{P}(V(G))$ is naturally obtained from the maps $\varphi_1$ and
$\varphi_2$ associated to $T(G_1)$ and $T(G_2)$ respectively, by
setting $\varphi(t)=\varphi_i(t)$ if $t\in V(T(G_i))$ for $i=1, 2$,
and $\varphi(r)=V(G)$.  We proceed as follows:

\begin{itemize}
\item  
If $G$ is not connected or not anticonnected, we use a component or
anticomponent decomposition.  
\item 
Otherwise, if $G$ has a non-trivial module $M$ we use the module
decomposition.  
\item 
Otherwise, if $G$ has a clique-cutset, we use a clique-cutset
decomposition.  
\item 
Finally if $G$ is nearly $\mathcal{C}'$, then we use the
$\mathcal{C'}$-antineighborhood decomposition.
\end{itemize}

Observe that, since prime atoms of $\mathcal{C}$ are nearly
$\mathcal{C}'$, then every $G\in \mathcal{C}$ satisfies at least one
condition, hence $T(G)$ is well-defined (but maybe not unique - this
does not matter).

We now want to define an injective $\mathcal{S}$-labeling of $T(G)$,
for some well-chosen $\mathcal{S}$.  A \emph{trio} of $G$ is a subset
$X\subseteq V(G)$ of at most three vertices, containing a non-edge.
Let $\mathcal{S}$ be the set of trios of $G$.  Clearly
$|\mathcal{S}|\le |V(G)|^3$.

Once again, we distinguish cases depending on the rule that was used
to decompose each internal node $t$.
\begin{itemize}
\item 
Rule 1: $G[\varphi(t)]$ is not connected, i.e. $\varphi(t)=V_1\uplus
V_2$ with $V_1$ anticomplete to $V_2$.  Let $v_i$ be any vertex of
$V_i$, for $i=1,2$ and define $\ell(t)=\{v_1, v_2\}$.
\item 
Rule 2: $G[\varphi(t)]$ is not anticonnected, i.e., $\varphi(t)=
V_1\uplus V_2$ with $V_1$ complete to $V_2$.  Since $t$ is not a
clique (otherwise, $t$ would be a leaf), there exists a non-edge $uv$
in, say, $V_1$.  Let $v_2$ be any vertex of $V_2$ and define
$\ell(t)=\{u, v, v_2\}$.
\item 
Rule 3: $G[\varphi(t)]$ has a non-trivial module $M$.  Let $u,v$ be
two vertices of $M$ and $x$ be a common non-neighbor of $u$ and $v$
(which exists otherwise Rule 2 applies).  We define $\ell(t)=\{u, v,
x\}$.
\item 
Rule 4: $G[\varphi(t)]$ has a partition $(V_1, K, V_2)$ where $K$ is a
clique-cutset that separates $V_1$ from $V_2$.  Let $v_i$ be any
vertex of $V_i$, for $i=1,2$ and define $\ell(t)=\{v_1, v_2\}$.
\item 
Rule 5: $G[\varphi(t)]$ has a $\mathcal{C}'$-antineighborhood
decomposition at vertex $x$: let $v$ be any non-neighbor of $x$ (which
exists, otherwise Rule 2 applies) and define $\ell(t)=\{x,v\}$.
\end{itemize}

For Rules 1 to 5, we have that $\ell(t) \ns \varphi(s)$ and $\ell(t)
\ns \varphi(s')$ where $s, s'$ are the children of $t$.  Moreover, for
Rules 1 up to 3, $\varphi(s)$ and $\varphi(s')$ intersect on at most
one vertex, and for Rule 4, $\varphi(s)\cap \varphi(s')$ is a clique.
Finally for Rule 5, the node $s'$ is a leaf.  Hence $\varphi(s)\cap
\varphi(s')$ does not contain any trio, unless $s'$ is a leaf.  Hence
$\ell$ satisfies both conditions of Lemma \ref{lem: injective
labeling}, consequently $\ell$ is injective.  This implies that
$|V(T(G))|$ is a $\mathcal{O}(|V(G)|^3)$.  By Lemma \ref{lem:valid
decompo}, every decomposition used above is valid so $T(G)$ is a valid
decomposition tree with leaves in $\mathcal{C}'\cup \mathcal{K}$.
Every graph $H$ of $\mathcal{C}'$ admits a CS-Separator of size
$\mathcal{O}(|V(H)|^c)$ by assumption, and every graph $H$ of
$\mathcal{K}$ is a clique so admits a linear CS-Separator.  According
to Lemma \ref{lem: CS-Sep from poly valid decompo tree}, $G\in
\mathcal{C}$ admits a CS-Separator of size
$\mathcal{O}(|V(G)|^{c+3})$.  This concludes the proof.
\end{proof}

We say that a class of graph $\C$ \emph{is at distance} $f(n)$ of a
class $\C'$ if for every member $G$ of $\C$ with $|V(G)| = n$ there
exists a set $D \subseteq V(G)$ such that $|D| \leq f(n)$ and $G
\setminus D \in \C'$.

\begin{lemma}\label{lem:distance}
Let $\C'$ be a class of graphs having the $\mathcal{O}(n^{c'})$-CS-Separation property,     and 
    let $\C$ be a class of graphs at distance $c \cdot \log(n)$ of $\C'$,
   for some positive constants $c,c'$.
    Then, $\C$ has the $O(n^{c'+c})$-CS-Separation property.
\end{lemma}
\begin{proof}
Let $G$ be a graph in $\C$ with $|V(G)| = n$ and $D \subseteq V(G)$ be a
set of at most $c \log(n)$ vertices such that $G \setminus D$
is in $\C'$.  Let $G' = G \setminus D$ be the subgraph induced by
$V(G) \setminus D$.  Since $G'$ is in $\C'$, it admits a CS-Separator
of  size $O(n^{c'})$ that separates any disjoint clique and
stable set in $G'$.  We build the family $\F$ of cuts  of $G$ as
follows.  For each cut $F' = (W, W')$ in $\F'$ and for each subset of
vertices $X \subseteq D$, add the cut $F = (W \cup X, W' \cup (D
\setminus X))$ to $\F$.  Let us check that $\F$ is a CS-separator of
$G$: for every clique $K$ and every stable set $S$ disjoint from $K$
in $G$, there is a cut in $\F' = (W, W')$ that separates $K \cap
V(G')$ and $S \cap V(G')$, furthermore,
by choosing $X=K\cap D$, the corresponding cut $F = (W \cup X, W'
\cup (D \setminus X))$ of $\F$ separates $K$ and $S$ in $G$.  The
family $\F$ of cuts  is of size $O(n^{c'} 2^{c \log n}) = O(n^{c' +
c})$.
\end{proof}

\section{Applications}
\label{sec:applications}

Let us now apply the method exposed above to prove that the
polynomial CS-Separation property holds in various classes of graphs.

\subsection{Apple-free graphs}
\label{subsec:apple-free}

An \emph{apple} $A_k$ is the graph obtained from a chordless cycle
$C_k$ of length at least $4$ by adding a vertex having exactly one
neighbor on the cycle (see Figure~\ref{fig:a4}).  A graph is
\emph{apple-free} if it does not contain any $A_k$ for $k\geq 4$ as an
induced subgraph.  This class of graphs was introduced in \cite{O}
under the name \emph{pan-free graphs} as a generalization of claw-free
graphs.  Brandst\"adt et al.~\cite{BLM} provide an in-depth structural
study of apple-free graphs in order to design a polynomial-time
algorithm for solving the maximum weighted stable set problem in this
class.
 
The polynomial CS-Separation property holds both for claw-free graphs and
chordal graphs (and even for any $C_4$-free graphs), so it seems interesting to
study the CS-Separation in the class of apple-free graphs, which generalizes
both these classes.  Moreover, it is not known whether the polynomial
CS-Separation property holds in perfect graphs.  Hence apple-free graphs appear
as a successful attempt to forbid cycle-like structures, especially if the
method seems to provide some insight on how to tackle the problem in further
classes. We can prove the following:

\begin{theorem}\label{thm:main}
The class of apple-free graphs has the polynomial CS-Separation property.
\end{theorem}
%
%
Let $A_4$, $A_5$, $D_6$ and $E_6$ be the graphs described in
Figure~\ref{fig:a4}.  For fixed $p \geq 3$, let $\ck{p}$ be the set of
chordless cycles on at least $p$ vertices.  Here are properties from
\cite{BLM} that we will use.
\begin{theorem}[\cite{BLM}]\label{th: structure apple-free}
    \
\begin{enumerate}[label=(\roman*)]
\item\label{item:1}
Every apple-free graph is either $(\ck{7})$-free or claw-free.

\item\label{item:2}
Every prime $(A_4, A_5, A_6, \ck{7})$-free atom is nearly $D_6$- and
$E_6$-free.

\item\label{item:3}
Every $(A_4, A_5, A_6, D_6, E_6, \ck{7})$-free graph is either
$C_6$-free or claw-free.

\item\label{item:4}
Every prime $(A_4, A_5, \ck{6})$-free atom is nearly $C_5$-free.

\item\label{item:5}
Every prime $(A_4, C_5, \ck{6})$-free atom is nearly chordal.
\end{enumerate}
\end{theorem}

\begin{figure}[ht]
    \begin{center}
        \includegraphics[scale=1.0]{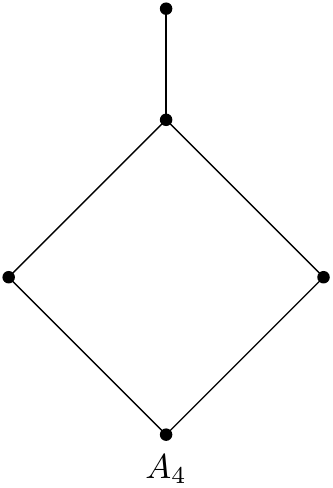}
        \hspace{1em}
        \includegraphics[scale=1.0]{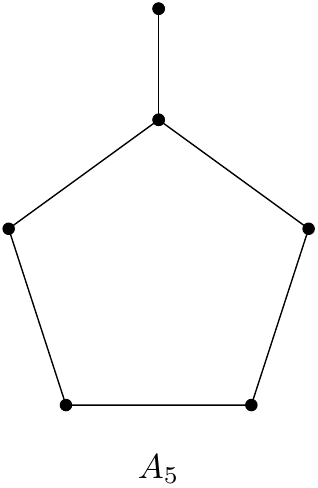}
        \hspace{1em}
        \includegraphics[scale=1.0]{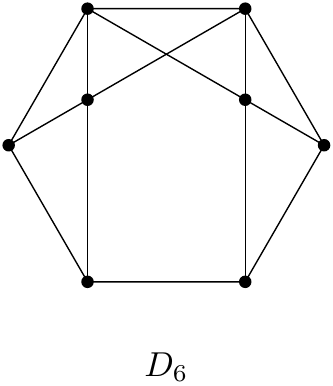}
        \hspace{1em}
        \includegraphics[scale=1.0]{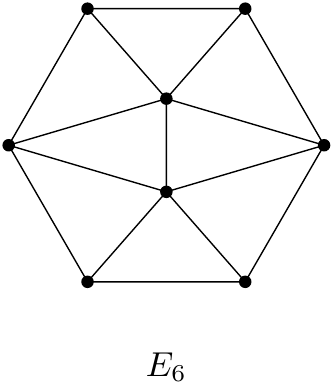}
    \end{center}
    \caption{From left to right: $A_4$, $A_5$, $D_6$ and $E_6$.}
    \label{fig:a4}
\end{figure}

%
\begin{proof}[Proof of Theorem \ref{thm:main}]
We repeatedly use Theorem~\ref{thm:prime-atoms-nearly-Cp} for each
item of Theorem~\ref{th: structure apple-free}.

By~\ref{item:5}, prime $(A_4, \ck{5})$-free atoms are nearly chordal.
Moreover chordal graphs have the linear-CS-Separation property
 (because a
chordal graph has linearly many maximal cliques), hence by
Theorem~\ref{thm:prime-atoms-nearly-Cp}, every $(A_4, \ck{5})$-free
graph on $n$ vertices admits a CS-Separator of size
$\mathcal{O}(n^{4})$.

Now, by~\ref{item:4}, prime $(A_4, A_5, \ck{6})$-free atoms are nearly
$C_5$-free, hence nearly $(A_4, \ck{5})$-free.  By
Theorem~\ref{thm:prime-atoms-nearly-Cp}, we obtain that every $(A_4,
A_5, \ck{6})$-free graph on $n$ vertices admits a CS-separator of size
$\mathcal{O}(n^7)$.  

By~\ref{item:3}, every $(A_4, A_5, A_6, D_6, E_6, \ck{7})$-free graph
$G$ is either (a) $C_6$-free, and in that case $G$ is $(A_4, A_5,
\ck{6})$-free and admits a CS-separator of size $\mathcal{O}(n^7)$,
according to the previous point; or (b) claw-free, in which case it
also admits a CS-separator of size $\mathcal{O}(n^4)$ (the polynomial
bound in \cite{BLT} is much worse but stated in a more general
context; it can be improved in the specific case of claw-free graphs).

Now, by~\ref{item:2}, prime $(A_4, A_5, A_6, \ck{7})$-free atoms are
nearly $D_6$- and $E_6$-free, so by
Theorem~\ref{thm:prime-atoms-nearly-Cp}, every $(A_4, A_5, A_6,
\ck{7})$-free graph on $n$ vertices admits a CS-separator of size
$\mathcal{O}(n^{10})$.  

Finally, by~\ref{item:1}, any apple-free graph $G$ on $n$ vertices is
either claw-free or $(A_4, A_5, A_6, \ck{7})$-free, so in any case it
admits a CS-separator of size $\mathcal{O}(n^{10})$.
\end{proof}

When the polynomial CS-Separation property holds in a class $\mathcal{C}$ of
graphs, it is natural to wonder whether the following stronger
condition also holds: 
does the Stable Set polytope admit a compact extended formulation in  $\mathcal{C}$?
  We
can give a negative answer to this question in apple-free graphs:
indeed, Rothvoss \cite{R} proved the existence of line graphs with no compact extended formulation. 
Since apple-free graphs contain all claw-free
graphs and in particular all line graphs, it follows that there is no compact extended formulation for apple-free graphs. In other words, the
extension complexity of the Stable Set polytope in this class
is not polynomial.

\subsection{Cap-free graphs}
\label{subsec:cap-free}

A \emph{hole} is a chordless cycle of length at least 4.  A \emph{cap}
is a hole with an additional vertex adjacent to exactly two consecutive
vertices in the hole.  A graph is \emph{cap-free} if it does not
contain any cap graph as an induced subgraph.  In the context of the
Strong Perfect Graph Conjecture, Conforti et~al.~\cite{CCK} studied
various classes of graphs, including the class of cap-free graphs, for
which they provide a decomposition theorem and a polynomial-time
recognition algorithm.  A \emph{basic cap-free graph} is either a
chordal graph or a 2-connected triangle-free graph with at most one
additional universal vertex, which we denote by (2-connected)
\emph{almost triangle-free graph}.  They proved the following:
\begin{theorem}[\cite{CCK}]\label{thm:decompo-cap}
Every connected cap-free graph either has an amalgam or is a basic
cap-free graph.
\end{theorem}
Recall that chordal graphs have the linear CS-separator property.
 Moreover, almost
triangle-free graphs have a quadratic number of maximal cliques (each edge
plus the universal vertex). So connected basic cap-free graphs have the quadratic-CS-Separation property.

We can finally prove the following:
%
%

\begin{theorem}
Cap-free graphs have the $\mathcal{O}(n^5)$-CS-Separation property.
\end{theorem}
\begin{proof}
We follow the method introduced in Section~\ref{sec:tree} by
recursively decomposing every member $G$ of $\mathcal{C}$ along
components, anticomponents and amalgams until reaching basic cap-free graphs.

For every $G\in\mathcal{C}$, we recursively define a valid decomposition tree
$T(G)$ with leaves in basic cap-free graphs as follows.  If
$G$ is a basic cap-free graph, then the root $r$ is the only node
of $T(G)$.  Otherwise, we use a valid decomposition $(G_1, G_2)$ of
$G$ as described below and $T(G)$ is obtained from $T(G_1)$ and
$T(G_2)$ by connecting the root $r$ to the respective roots of
$T(G_1)$ and $T(G_2)$.  Moreover the map $\varphi:V(T)\to
\mathcal{P}(V(G))$ is naturally obtained from the maps $\varphi_1$ and
$\varphi_2$ associated to $T(G_1)$ and $T(G_2)$ respectively, by
setting $\varphi(t)=\varphi_i(t)$ if $t\in V(T(G_i))$ for $i=1, 2$,
and $\varphi(r)=V(G)$.  We proceed as follows:

\begin{itemize}
\item  
If $G$ is not connected, we use a component decomposition.  
\item If $G$ is not anticonnected, we use an anticomponent decomposition.  
\item 
Otherwise, if $G$ has an amalgam, we use an amalgam
decomposition. 
\end{itemize}

By Theorem~\ref{thm:decompo-cap}, every connected cap-free graph either has an amalgam 
or is a basic cap-free graph. Hence $T(G)$ is well-defined (but maybe not unique - this
does not matter).
We define an injective $\mathcal{S}$-labeling of $T(G)$,
for some well-chosen $\mathcal{S}$.
Let $\mathcal{S}$ be the set of trios of $G$ (as in the proof of 
Theorem~\ref{thm:prime-atoms-nearly-Cp}, a \emph{trio} is a set of at most three vertices containing a non-edge).  Clearly $|\mathcal{S}|\le |V(G)|^3$.

Once again, we distinguish cases depending on the rule that was used
to decompose each internal node $t$.
\begin{itemize}
\item 
Rule 1: $G[\varphi(t)]$ is not connected, i.e. $\varphi(t)=V_1\uplus
V_2$ with $V_1$ anticomplete to $V_2$.  Let $v_i$ be any vertex of
$V_i$, for $i=1,2$ and define $\ell(t)=\{v_1, v_2\}$.
\item 
Rule 2: $G[\varphi(t)]$ is not anticonnected, i.e., $\varphi(t)=
V_1\uplus V_2$ with $V_1$ complete to $V_2$.  Since $\varphi(t)$ is not a
clique (otherwise, $t$ would be a leaf), there exists a non-edge $uv$
in, say, $V_1$.  Let $v_2$ be any vertex of $V_2$ and define
$\ell(t)=\{u, v, v_2\}$.
\item 
Rule 3:  $G[\varphi(t)]$  contains an amalgam, i.e., $\varphi(t)=
C \uplus A_1 \uplus A_2 \uplus B_1 \uplus B_2$ 
with $C$ a clique complete to $A_1 \cup A_2$, $A_1$ complete to $A_2$, 
$B_1$ anticomplete to $A_2 \cup B_2$ and symmetrically $B_2$ anticomplete 
to $A_1 \cup B_1$.  Note that $B_1 \cup B_2 \neq \emptyset$ for otherwise, 
$G[\varphi(t)]$ would not be anticonnected and we would have applied Rule 2 instead.  So we may assume that $B_1$ contains at least one vertex $b_1$. By definition of an amalgam decomposition, there exists at least two vertices $u_2, v_2\in A_2\cup B_2$. Then we define  $\ell(t) = \{b_1, u_2, v_2\}$.
\end{itemize}

For Rules 1 to 3, we have that $\ell(t) \ns \varphi(s)$ and $\ell(t)
\ns \varphi(s')$ where $s, s'$ are the children of $t$.  Moreover, for
Rules 1 and 2, $\varphi(s)$ and $\varphi(s')$ do not intersect, and 
for Rule 3, $\varphi(s)\cap \varphi(s')$ is a clique. Hence $\varphi(s)\cap
\varphi(s')$ does not contain any trio.  Hence
$\ell$ satisfies both conditions of Lemma \ref{lem: injective
labeling}, consequently $\ell$ is injective.  This implies that
$|V(T(G))|$ is a $\mathcal{O}(n^3)$.  
Moreover, for every leaf $f$, $G[\varphi(f)]$ is a basic cap-free graph, hence it
 admits a CS-Separator of size
$\mathcal{O}(n^2)$.  According
to Lemma \ref{lem: CS-Sep from poly valid decompo tree}, $G\in
\mathcal{C}$ admits a CS-Separator of size
$\mathcal{O}(n^5)$.  
This concludes the proof.
\end{proof}

\subsection{Diamond-wheel-free and $k$-windmill-free graphs}
Although it was already known that both diamond-wheel-free-graphs and
$k$-windmill-free graphs have the polynomial CS-Separation property, as proved in
\cite[sections 2.3.2 and 3.4.3]{La}, it is worth noting that the
general framework discussed in this paper works just fine with those
classes of graphs, using $\mathcal{C'}$-neighborhood decomposition with some well-chosen class $\mathcal{C'}$.

\section{Limits of the graph decompositions}\label{sec:limit}


In this section, contrary to previous sections, we present a negative result.
More precisely, we highlight that some decomposition theorems
have no chance to be useful for proving the polynomial CS-Separation property.
This happens when the decomposition theorem is based only on one classical decomposition called star-cutset decomposition.

A \emph{star-cutset} in a graph $G$ is a vertex cut $X$ such that
there is a vertex $x\in X$ that is adjacent to every vertex in
$X\setminus\{x\}$.  Let $V_1, V_2$ be a partition of $G[V \setminus
X]$ such that $V_1$ is anticomplete to $V_2$.  Then (as usual) $(G[V_1
\cup X], G[V_2 \cup X])$ is a cutset decomposition of $G$, and in this
case we call it a \emph{star-cutset} decomposition.

\begin{theorem}\label{thm:star-cutset}
There exists a class $\mathcal{D}$ of graphs such that:
\begin{enumerate}
\item every graph of $\mathcal{D}$ is either a clique or admits a star-cutset decomposition $(G_1, G_2)$ with $G_1, G_2\in \mathcal{D}$, and
\item $\mathcal{D}$ does not have the polynomial CS-Separator property.
\end{enumerate}
\end{theorem}

\begin{proof}
Let $\mathcal{C}$ be a class of graphs that does not have the polynomial
CS-Separation property.  Such a class of graphs exist by~\cite{G} (it may be the class of all graphs). 
Let us now denote by $\mathcal{C}'$ the closure of $\mathcal{C}$ by taking
induced subgraphs, then we define
  $$\mathcal{D}=\{H=(V,E) \ |\  \exists x\in V(H) \text{ such that $x$ is a universal vertex and } H[V\setminus x]\in \mathcal{C'}\}$$
 Let us prove that $\mathcal{D}$ satisfies the first item.
Let $G$ be a graph of $\mathcal{D}$ that is not a clique.  Let $u,v$
be two non-adjacent vertices.  Let $x$ be a universal vertex of $G$, which exists by definition of $\mathcal{D}$. 
 Note that $(x \cup N(x)) \setminus
\{u,v \} = V \setminus \{ u,v \}$ is a star-cutset.  So one can
decompose $G$ into $G_1= G[V \setminus u]$ and $G_2=G[V \setminus v]$.
By definition of $\mathcal{D}$, both graphs $G_1$ and $G_2$ are in $\mathcal{D}$ because $x$
remains a universal vertex in $G_1, G_2$, and $G_1\setminus x, G_2\setminus x$ are induced subgraphs of $G\setminus x \in \mathcal{C'}$.

Moreover, we prove by contradiction that $\mathcal{D}$ satisfies the
second item.  Assume that $\mathcal{D}$ has the polynomial
CS-Separation property, for some polynomial $P$.  Let $G\in
\mathcal{C}$, then $G\in \mathcal{C'}$ so the graph $H$ obtained from
$G$ by adding a universal vertex $x$ is in $\mathcal{D}$.  In
particular, it admits a CS-Separator $\mathcal{F}$ of size
$P(|V(H)|)$, and replacing every cut $(B,W)$ of $\mathcal{F}$ by
$(B\setminus x, W\setminus x)$ provides a CS-Separator for $G$ of the
same size.  Hence $\mathcal{C}$ has the polynomial CS-Separation
property, a contradiction.
\end{proof}

A drawback of this theorem is that $\mathcal{D}$ is not hereditary.
It would be nice to know whether there a exists a hereditary class of
graphs satisfying the condition of Theorem \ref{thm:star-cutset}.



\end{document}